\documentclass[english]{amsart}
\usepackage{amsmath}
\usepackage{amssymb}
\usepackage{lineno}
\usepackage{graphics}
\usepackage[T1]{fontenc}
\usepackage[latin9]{inputenc}
\usepackage{babel}
\usepackage{amsbsy}
\usepackage{amstext}
\usepackage{amsthm}
\usepackage{enumitem}
\usepackage[unicode=true,
 bookmarks=true,bookmarksnumbered=false,bookmarksopen=false,
 breaklinks=false,pdfborder={0 0 1},backref=false,colorlinks=false]
 {hyperref}
\usepackage{breakurl}
\usepackage{mathtools}
\usepackage[mathscr]{eucal}
\renewcommand{\labelenumi}{(\arabic{enumi})}
\makeatletter
\@namedef{subjclassname@2020}{%
  \textup{2020} Mathematics Subject Classification}
\makeatother

\makeatletter

  \newtheorem{theorem}{Theorem}[section]
  \newtheorem{corollary}[theorem]{Corollary}
  \newtheorem{lemma}[theorem]{Lemma}
  \newtheorem{question}[theorem]{Question}
  \newtheorem{proposition}[theorem]{Proposition}
  \newtheorem{remark}[theorem]{Remark}
  
  \theoremstyle{definition}
  \newtheorem{definition}[theorem]{Definition}
  
\begin{document}

\title{D- Sets in Arbitrary Semigroup}

\author{Surajit Biswas}
\address{Ramakrishna Mission Vidyamandira, Belur Math, Howrah-711202, India.}
\email{241surajit@gmail.com}

\author{Bedanta Bose}
\address{Swami Niswambalananda Girls' College, Uttarpara, Hooghly-712232, India.}
\email{ana\textunderscore bedanta@yahoo.com}

\author{Sourav Kanti Patra}
\address{Indian Institute of Science Education and Research Berhampur, Ganjam District, Odisha-760010, India.}
\email{souravkantipatra@gmail.com}

\begin{abstract}
We define the notion of $D$-set in an arbitrary semigroup, and with some mild restrictions we establish its dynamical and combinatorial characterizations. Assuming a weak form of cancellation in semigroups we have shown that the Cartesian product of finitely many $D$-sets is a $D$-set. A similar partial result has been proved for Cartesian product of infinitely many $D$-sets. Finally, in a commutative semigroup we deduce that $D$-sets (with respect to a F{\o}lner net) are $C$-sets.
\end{abstract}



\maketitle

\section{Introduction}
The notion of $D$-set in $\mathbb{N}$ first appeared in \cite{bbdf} by Beiglb\"{o}ck, Bergelson, Downarowicz and Fish. They showed that any Rado system (a finite partition regular system of homogeneous linear equations with integer coefficients) is solvable in $D$-sets. This result itself makes the collection of $D$-sets an interesting family, as it extends the family of central sets and solutions to Rado systems can be found in every central set \cite[Theorem 8.22]{f}. In fact, $D$-sets satisfy the conclusion of the Central Sets Theorem, i.e. they are $C$-sets \cite[Theorem 11]{bbdf}.\par
In \cite{bbdf}, a set $A\subseteq\mathbb{N}$ is called a $D$-set if there exists a compact dynamical system $(X,T)$ (i.e. a compact metric space $X$ and a continuous transformation $T$ on $X$), a pair of points $x,y\in X$ where $y$ is essentially recurrent\footnote{Given a dynamical system $(X,T),\;y\in X$ is called essentially recurrent if the set $\{n\in\mathbb{N}:\,T^n y\in U\}$ has positive upper Banach density for every neighborhood $U$ of $y$. Although, in Theorem \ref{Theorem 1.1} when we say $y\in X$ is essentially recurrent, we mean $\{n\in\mathbb{Z}:T^n y\in U\}$ has positive upper Banach density for every neighborhood $U$ of $y$.}, and such that $(y,y)$ belongs to the orbit closure of $(x,y)$ in the product system $(X\times X,T\times T)$, and an open neighborhood $U_y$ of $(y,y)$ such that $A=\{n\in\mathbb{N}:\,(T^n x,T^n y)\in U_y\}$. However, in \cite[Definition 1.2]{bd}, the collection of $D$-sets in $\mathbb{Z}$ is defined algebraically as the union of all idempotents $p\in \beta \mathbb{Z}$ ($\beta \mathbb{Z}$ is the Stone-\v{C}ech compactification of $\mathbb{Z}$) such that every member of $p$ has positive upper Banach density\footnote{The upper Banach density of $A\subset \mathbb{Z}$ is defined by $\limsup_{m-n\rightarrow \infty}\frac{|A\cap[n,m-1]|}{m-n}$.}. But \cite[Theorem 2.8]{bd} provides the following equivalent dynamical characterization of $D$-set in $\mathbb{Z}$.

\begin{theorem}\label{Theorem 1.1}
A set $A\subseteq \mathbb{Z}$ is a $D$-set if and only if there exists a compact dynamical system $(X,T)$ where $T$ is a homeomorphism, a pair of points $x,y\in X$ where $y$ is essentially recurrent, and such that $(y,y)$ belongs to the orbit closure of $(x,y)$ in the product system $(X\times X,T\times T)$, and an open neighborhood $U_y$ of $(y,y)$ such that $A=\{n\in\mathbb{Z}:\,(T^n x,T^n y)\in U_y\}$.
\end{theorem}

As mentioned in \cite{bbdf}, the above equivalence holds for the semigroup $(\mathbb{N},+)$ as well if we define $D$-set in $\mathbb{N}$ similarly as for $\mathbb{Z}$. 

\begin{definition}\label{Definition 1.2}
The collection $\mathcal{D}$ (of $D$-sets in $\mathbb{N}$) is the union of all idempotents $p\in \beta \mathbb{N}$ ($\beta \mathbb{N}$ stands for the Stone-\v{C}ech compactification of $\mathbb{N}$) such that every member of $p$ has positive upper Banach density.
\end{definition}

In Section \ref{Section 2} of this article we define an algebraic notion of $D$-set in an arbitrary semigroup (Definition \ref{Definition 2.2}) and show its existence in a large class of semigroups (Theorem \ref{Theorem 2.5}). In Section \ref{Section 3} we provide a dynamical characterization of $D$-sets (Theorem \ref{Theorem 3.7}) with some restrictions which we will introduce accordingly. In Section \ref{Section 4} we provide a combinatorial characterization (Theorem \ref{Theorem 4.5}) which shows the combinatorial richness of $D$-sets. In \cite{hs-1}, Hindman and Strauss have shown that the notions of central sets, $C$-sets and $J$-sets are preserved under finite Cartesian products. In Section \ref{Section 5} we prove the same for $D$-sets (Theorem \ref{Theorem 5.8}) assuming some weak form of cancellation in the semigroups. Moreover, in Corollary \ref{Corollary 5.14} we deduce that the projections from an infinite product of semigroups maps $D$-sets onto $D$-sets. In Theorem \ref{Theorem 5.18} we provide a partial converse of this result. In Section \ref{Section 6} it is shown that for a class of nets that exist in many semigroups, including all commutative and cancellative semigroups, any central set is a $D$-set (Proposition \ref{Proposition 6.2}). Moreover, in Theorem \ref{Theorem 6.13} we deduce that in a commutative semigroup, any $D$-set (with respect to a F{\o}lner net) is a $C$-set. Note that $C$-sets are precisely the sets those satisfies the conclusion of the Central Sets Theorem \cite[Corollary
14.14.10]{hs}, and so are guaranteed to have substantial combinatorial structure. \par
Our definition of $D$-set in an arbitrary semigroup $S$ is motivated from the algebraic definition of $D$-sets in $\mathbb{N}$ (Definition \ref{Definition 1.2}), and it is in the setting of $\beta{S}$, where $\beta{S}$ is the Stone-\v{C}ech compactification of the semigroup $S$ \cite[Definition 3.25]{hs}. From now onwards we will assume all our semigroups to be discrete. For a discrete semigroup $S$, $\beta S$ can be naturally identified with the collection of all ultrafilters on $S$ \cite[Theorem 3.27]{hs}, and the semigroup structure on $S$ induces a unique semigroup structure on $\beta{S}$ so that $\beta{S}$ becomes a right topological semigroup with $S$ contained in its topological center \cite[Theorem 4.1]{hs}. This semigroup structure is explicitly given by, for $p,q\in\beta{S},\,pq:=\{A\subseteq S:\,\{s\in S:\,s^{-1}A\in q\}\in p\}$. As a compact right topological semigroup, $\beta{S}$ has a smallest two-sided ideal denoted $K(\beta{S})$, which is the union of all minimal right ideals of $\beta{S}$ and is also the union of all minimal left ideals \cite[Theorem 2.8]{hs}. For other details regarding the semigroup $\beta{S}$, see \cite{hs}.

\section{Preliminaries}\label{Section 2}
We will use the following notion of upper Banach density \cite[Definition 2.1(c)]{hs-2} to define $D$-sets in arbitrary semigroup. (Given a set $X$ we write $\mathcal{P}_f(X)$ for the set of finite nonempty subset of $X$.)

\begin{definition}
Let $S$ be a semigroup, let $\mathcal{F}=\langle F_i\rangle_{i\in I}$ be a net in $\mathcal{P}_f(S)$, and let $A\subseteq S$. Then the upper Banach density of $A$ with respect to $\mathcal{F}$ is $$d_{\mathcal{F}}^*(A)=sup\{\alpha: (\forall i_0\in I)(\exists i\geq i_0)(\exists s\in S\cup \{1\})(|A\cap(F_i\cdot s)|\geq \alpha |F_i|)\}.$$
\end{definition}

We are not assuming that $S$ has an identity. By the notation $F_i\cdot 1\,(i\in I)$, we simply mean $F_i$. Our next definition provides the algebraic notion of $D$-sets in arbitrary semigroup.

\begin{definition}\label{Definition 2.2}
Let $S$ be a semigroup, let $\mathcal{F}$ be a net in $\mathcal{P}_f(S)$. Let $D_{\mathcal{F}}^*=\{p\in\beta S:(\forall A\in p)(d_{\mathcal{F}}^*(A)>0)\}$. Then a subset $A$ of $S$ is a $D$-set with respect to $\mathcal{F}$ if it is contained in an idempotent of $D_{\mathcal{F}}^*$.
\end{definition}

Note that given a semigroup $S$ and a net $\mathcal{F}$ in $\mathcal{P}_f(S)$, $D_{\mathcal{F}}^*$ is a nonempty subset of $\beta S$ \cite[Lemma 2.3]{hs-2}. Moreover $D_{\mathcal{F}}^*$ is closed as follows. If $p\in \beta S \setminus D_{\mathcal{F}}^*$, then $d_{\mathcal{F}}^*(A)=0$ for some $A\in p$. Hence $\overline{A}:=\{q\in \beta S: A\in q\}$ is an open neighborhood of $p$ and $\overline{A}\subseteq \beta S \setminus D_{\mathcal{F}}^*$. Thus $\beta S \setminus D_{\mathcal{F}}^*$ is open, i.e. $D_{\mathcal{F}}^*$ is closed. In \cite[Theorem 2.4]{hs-2} it is shown that a weak form of right cancellation in $S$ makes $D_{\mathcal{F}}^*$ a right ideal of $\beta S$. On the other hand, in \cite[Theorem 2.7]{hs-2} it is shown that the following $(*)$ condition on the net $\mathcal{F}$ and a weak form of left cancellation in $S$ makes $D_{\mathcal{F}}^*$ a left ideal of $\beta S$. Here we recall the terminologies for these weak form of cancellations \cite[Definition 1.7]{hs-2}, along with the above mentioned $(*)$ condition \cite[Definition 2.6]{hs-2}.

\begin{definition}
Let $S$ be a semigroup and let $b\in\mathbb{N}$. Then $S$ is $b$-weakly left cancellative (respectively $b$-weakly right cancellative) if for all $t,u\in S,\,|\{s\in S:\,ts=u\}|\leq b$ (respectively $|\{s\in S:\,st=u\}|\leq b$).
\end{definition}

\begin{definition}
Let $S$ be a semigroup and let $\mathcal{F}=\langle F_i \rangle_{i\in I}$ be a net in $\mathcal{P}_f(S)$. Then $\mathcal{F}$ might satisfy the folowing property.
\begin{enumerate}
\item[$(*)$] $(\forall \epsilon >0)(\forall t   \in S)(\exists c\in \mathbb{N})(\exists l \in I)(\forall i_0\geq l)(\exists i\geq i_0)(\exists s\in S\cup \{1\})\\
(|(t\cdot F_{i_0})\setminus (F_i\cdot s)|<\epsilon\cdot |F_{i_0}| \text{ and } |F_i|\leq c\cdot |F_{i_0}|)$.
\end{enumerate}
\end{definition}

Now the following theorem shows the existence of $D$-sets in a large class of semigroups. 

\begin{theorem}\label{Theorem 2.5}
Let $S$ be a semigroup with a given net $\mathcal{F}$ in $\mathcal{P}_f(S)$. Let either of the following two conditions hold:
\begin{enumerate}
\item There is some $b\in\mathbb{N}$ such that $S$ is $b$-weakly right cancellative.
\item There is some $b\in\mathbb{N}$ such that $S$ is $b$-weakly left cancellative, and $\mathcal{F}$ satisfies the $(*)$ condition.
\end{enumerate}
Then a $D$-set with respect to $\mathcal{F}$ exists in $S$.
\end{theorem}

\begin{proof}
The condition $\mathit{(1)}$ makes $D_{\mathcal{F}}^*$ a closed right ideal of $\beta S$, whereas the condition $\mathit{(2)}$ makes $D_{\mathcal{F}}^*$ a closed left ideal of $\beta S$. Thus either of the conditions $\mathit{(1)}$ or $\mathit{(2)}$ in particular makes $D_{\mathcal{F}}^*$ a compact subsemigroup of $\beta S$. Now, as $\beta S$ is a right topological semigroup, so $D_{\mathcal{F}}^*$ becomes a compact right topological semigroup. Hence, Ellis' theorem \cite[Theorem 2.5]{hs} ensures the existence of an idempotent in $D_{\mathcal{F}}^*$, and therefore a $D$-set with respect to $\mathcal{F}$ exists in $S$. 
\end{proof}

In the rest of the article, when the net $\mathcal{F}$ is clear from the context, we shall not mention it and simply write $D$-set for \lq $D$-set with respect to $\mathcal{F}$ \rq.

\section{Dynamical characterization of D-sets}\label{Section 3}

In \cite{j,p,sy} the authors established dynamical characterizations of $C$-sets and central sets in arbitrary semigroup along with their near zero version. Our dynamical characterization of $D$-sets is motivated by these dynamical characterizations. Let us begin by recalling some terminology \cite[Definition 2.1(a), 2.6 and 3.1]{j} used later in this section.

\begin{definition}
Let $S$ be a nonempty discrete space and $\mathcal{K}$ a filter on $S$. Then $\overline{\mathcal{K}}=\{p\in\beta S: \mathcal{K}\subseteq p\}$.
\end{definition}

\begin{definition}
A pair $(X,\langle T_s\rangle_{s\in S})$ is a dynamical system if it satisfies the following four conditions:
\begin{enumerate}
\item $X$ is a compact Hausdorff space.
\item $S$ is a semigroup.
\item $T_s:X\rightarrow X$ is continuous for every $s\in S$.
\item For every $s,t\in S$ we have $T_{st}=T_s\circ T_t$.
\end{enumerate}
\end{definition}

\begin{definition}
Let $(X,\langle T_{s}\rangle_{s\in S})$ be a dynamical system, $x$ and $y$ points in $X$, and $\mathcal{K}$ a filter on $S$. The pair $(x,y)$ is jointly $\mathcal{K}$-recurrent if for every neighborhood $U$ of $y$ we have $S\setminus\{s\in S:T_{s}(x)\in U\text{ and }T_{s}(y)\in U\}\notin\mathcal{K}$.
\end{definition}

The key ingredient to produce our dynamical characterization of $D$-sets is the following result \cite[Theorem 3.3]{j}.

\begin{theorem}\label{Theorem 3.4}
Let $S$ be a semigroup, let $\mathcal{K}$ be a filter on $S$ such that $\mathcal{\overline{\mathcal{K}}}$ is a compact subsemigroup of $\beta S$, and let $A\subseteq S$. Then $A$ is a member of an idempotent in $\overline{\mathcal{K}}$ if and only if there exists a dynamical system $(X,\langle T_s\rangle_{s\in S})$ with points $x$ and $y$ in $X$ and there exists a neighborhood $U_{y}$ of $y$ such that the pair $(x,y)$ is jointly $\mathcal{K}$-recurrent and $A=\{s\in S:T_{s}(x)\in U_{y}\}$.
\end{theorem}

The following lemma directly follows from \cite[Theorem 2.2(a)]{j}. But here we provide a self-contained proof.

\begin{lemma}\label{Lemma 3.5}
Let $S$ be a semigroup with a given net $\mathcal{F}$ in $\mathcal{P}_f(S)$. Let $\mathcal{K}=\{A\subseteq S:d_{\mathcal{F}}^{*}(S\setminus A)=0\}$. Then $\mathcal{K}$ is a filter on $S$ and $\overline{\mathcal{K}}=D_{\mathcal{F}}^{*}$. Equivalently, $D_{\mathcal{F}}^{*}=\bigcap_{A\in\mathcal{K}}\overline{A}$. Moreover, let either of the following two conditions hold:
\begin{enumerate}
\item There is some $b\in\mathbb{N}$ such that $S$ is $b$-weakly right cancellative.
\item There is some $b\in\mathbb{N}$ such that $S$ is $b$-weakly left cancellative, and $\mathcal{F}$ satisfies the $(*)$ condition.
\end{enumerate}
If $A$ is a subset of $S$ such that $d_{\mathcal{F}}^{*}(S\setminus A)=0$,
then $A$ is a $D$-set with respect to $\mathcal{F}$.
\end{lemma}

\begin{proof}
By \cite[Definition 3.1]{hs}, $\mathcal{K}$ is a filter on $S$ if it is nonempty and satisfies the following three conditions:
\begin{enumerate}[label=(\alph*)]
\item If $A,B\in\mathcal{K}$, then $A\cap B\in\mathcal{K}$.
\item If $A\in\mathcal{K}$ and $A\subseteq B\subseteq S$, then $B\in\mathcal{K}$.
\item $\emptyset\notin\mathcal{K}$. 
\end{enumerate}
Now to prove (a), let $A,B\in\mathcal{K}$. Then $d_{\mathcal{F}}^*(S\setminus A)=d_{\mathcal{F}}^*(S\setminus B)=0$. Hence $d_{\mathcal{F}}^*\big(S\setminus(A\cap B)\big)=d_{\mathcal{F}}^*\big((S\setminus A)\cup (S\setminus B)\big)\leq d_{\mathcal{F}}^*(S\setminus A)+d_{\mathcal{F}}^*(S\setminus B)$ and we have $d_{\mathcal{F}}^*\big(S\setminus(A\cap B)\big)=0$, i.e. $A\cap B\in\mathcal{K}$. To prove (b), let $A\in\mathcal{K}$ and $A\subseteq B\subseteq S$. Then $S\setminus B\subseteq S\setminus A$ and we have $d_{\mathcal{F}}^*(S\setminus B)\leq d_{\mathcal{F}}^*(S\setminus A)$. Which implies $d_{\mathcal{F}}^*(S\setminus B)=0$, i.e. $B\in\mathcal{K}$. (c) is obvious as $d_{\mathcal{F}}^*(S)=1$. Thus $\mathcal{K}$ is a filter on $S$.\par 
To show $\overline{\mathcal{K}}=D_{\mathcal{F}}^*$, let $p$ be an element of  $\beta S \setminus D_{\mathcal{F}}^*$. Then there exists $A\in p$ such that $d_{\mathcal{F}}^*(A)=0$, i.e. $S\setminus A \in \mathcal{K}$. As $p$ is an ultrafilter, so by \cite[Theorem 3.6]{hs}, $A\in p$ implies $S\setminus A\notin p$. Therefore we have $S\setminus A \in \mathcal{K}\setminus p$, i.e. $\mathcal{K}\not\subseteq p$, i.e. $p\in \beta S \setminus \overline{\mathcal{K}}$. Thus we have $\beta S \setminus D_{\mathcal{F}}^*\subseteq \beta S \setminus \overline{\mathcal{K}}$. The above implications also holds in the reverse order and that will provide us the reverse inclusion $\beta S \setminus \overline{\mathcal{K}}\subseteq \beta S \setminus D_{\mathcal{F}}^*$. Taking these inclusions together, we get $\beta S \setminus \overline{\mathcal{K}} = \beta S \setminus D_{\mathcal{F}}^*$, i.e. $\overline{\mathcal{K}}=D_{\mathcal{F}}^*$. But $\overline{\mathcal{K}}=\bigcap_{A\in\mathcal{K}}\overline{A}$ and equivalently we have $D_{\mathcal{F}}^*=\bigcap_{A\in\mathcal{K}}\overline{A}$.\par 
For the last part, note that if condition $\mathit{(1)}$ or $\mathit{(2)}$ holds, then $D_{\mathcal{F}}^*$ is a compact subsemigroup of $\beta{S}$. Now as  $D_{\mathcal{F}}^{*}=\bigcap_{A\in \mathcal{K}}\overline{A}$, for any idempotent $p$ in $D_{\mathcal{F}}^{*}$, we have $p\in\bigcap_{A\in\mathcal{K}}\overline{A}$, i.e. $A\in p$ for any subset $A$ of $S$ such that $d_{\mathcal{F}}^{*}(S\setminus A)=0$. Hence, the last conclusion follows.
\end{proof}

\begin{lemma}\label{Lemma 3.6}
Let $S$ be a semigroup with a given net $\mathcal{F}$ in $\mathcal{P}_f(S)$. Let $\mathcal{K}=\{A\subseteq S:d_{\mathcal{F}}^{*}(S\setminus A)=0\}$. Let $(X,\langle T_{s}\rangle_{s\in S})$ be a dynamical system. Then a pair of points $(x,y)$ in $X\times X$ is jointly $\mathcal{K}$-recurrent if and only if for any neighborhood $U$ of $y$, $d_{\mathcal{F}}^{*}(\{s\in S:T_{s}(x)\in U\text{ and }T_{s}(y)\in U\})>0$.
\end{lemma}

\begin{proof}
$(x,y)$ is jointly $\mathcal{K}$-recurrent if and only if for any neighborhood $U$ of $y,\,S\setminus\{s\in S: T_s(x)\in U \text{ and }T_s(y)\in U\}\notin \mathcal{K}$. But $S\setminus\{s\in S:T_{s}(x)\in U \text{ and } T_{s}(y)\in U\}\notin\mathcal{K}$ if and only if  $d_{\mathcal{F}}^{*}(\{s\in S:T_{s}(x)\in U \text{ and } T_{s}(y)\in U\})>0$. 
\end{proof}

Finally we have the following dynamical characterization of $D$-sets in a large class of semigroups.

\begin{theorem}\label{Theorem 3.7}
Let $S$ be a semigroup with a given net $\mathcal{F}$ in $\mathcal{P}_f(S)$. Let either of the following two conditions hold:
\begin{enumerate}
\item There is some $b\in\mathbb{N}$ such that $S$ is $b$-weakly right cancellative.
\item There is some $b\in\mathbb{N}$ such that $S$ is $b$-weakly left cancellative, and $\mathcal{F}$ satisfies the $(*)$ condition.
\end{enumerate}
Then $A\subseteq S$ is a $D$-set with respect to $\mathcal{F}$ if and only if there exists a dynamical system $(X,\langle T_{s}\rangle_{s\in S})$ with points $x$ and $y$ in $X$ such that for any neighborhood $U$ of $y$, $d_{\mathcal{F}}^{*}(\{s\in S:T_{s}(x)\in U\text{ and }T_{s}(y)\in U\})>0$ and for some neighborhood $U_{y}$ of $y$, $A=\{s\in S:T_{s}(x)\in U_{y}\}$.
\end{theorem}

\begin{proof}
Note that by Lemma \ref{Lemma 3.5}, we have $D_{\mathcal{F}}^*=\overline{\mathcal{K}}$ where $\mathcal{K}=\{A\subseteq S:d_{\mathcal{F}}^*(S\setminus A)=0\}$. So Theorem \ref{Theorem 3.4} provides us the desired dynamical characterization of $D$-sets when we replace the condition of \lq joint $\mathcal{K}$-recurrence\rq \,by its equivalence (in terms of density) which we proved in Lemma \ref{Lemma 3.6}.
\end{proof}

\section{Combinatorial Characterization of D-sets}\label{Section 4}

In \cite[Theorem 3.8]{hms}, Hindman, Maleki and Strauss provide a combinatorial characterization of central sets. This ensures a sufficient amount of combinatorial richness in central sets. Also in \cite{bcp,hl} the authors provide combinatorial characterizations of the near zero version of $C^*$-sets and central sets. Being motivated by this, in this section we provide a combinatorial characterizations of $D$-sets in an arbitrary semigroup. For this purpose, we will recall the notion of "tree" \cite[Definition 3.4 and 3.5]{hms} in the next two definitions. We write $\omega=\{0,1,2,3,...\}$, the first infinite ordinal and note that each ordinal is defined by the set of its predecessors (e.g. $2=\{0,1\},\,0=\emptyset$ and if $f$ is the function $\{(0,0),(1,4),(2,1),(3,10)\}$, then $f|_3=\{(0,0),(1,4),(2,1)\}$).

\begin{definition}
$\mathsf{T}$ is a tree in $A$ if $\mathsf{T}$ is a set of functions and for each $f\in \mathsf{T},\,\text{domain\ensuremath{(f)}\ensuremath{\in\omega}\ }$
and range$(f)\subseteq A$ and if domain$(f)=n>0$, then $f|_{n-1}\in \mathsf{T}$. $\mathsf{T}$ is a tree if for some $A,\,\mathsf{T}$ is a tree in $A$.
\end{definition}

\begin{definition}
\begin{enumerate}
\item Let $f$ be a function with $domain(f)=n\in \omega$ and let $s$ be given. Then $f^{\frown} s=f\cup\{(n,s)\}$.
\item Given a tree $\mathsf{T}$ and $f\in \mathsf{T}, B_f=B_f(\mathsf{T})=\{s:f^{\frown} s \in \mathsf{T}\}$.
\item Let $S$ be a semigroup and let $A\subseteq S$. Then $\mathsf{T}$ is a $*$-tree in $A$ if $\mathsf{T}$  is a tree in $A$ and for all $s\in B_f,B_{f^{\frown} s}\subseteq s^{-1}B_f$.
\item Let $S$ be a semigroup and let $A\subseteq S$. Then $\mathsf{T}$ is a $FP$-tree in $A$ if $\mathsf{T}$ is a tree in $A$ and for all $f\in \mathsf{T}$ and for all $s\in B_f$,
$$B_f=\Big\{{\prod}_{m\in F}g(m):g\in \mathsf{T},f\subsetneq g, \text{ and } \emptyset\neq F\subseteq dom(g)\setminus dom(f)\Big\}.$$
\end{enumerate}
\end{definition}

Given a semigroup $S$, if $p$ is an idempotent in $\beta S$, then each element of $p$ naturally corresponds to a $*$-tree. This correspondence follows from the next two lemmas \cite[Lemma 3.6]{hms} and \cite[Lemma 4.6]{hl}.

\begin{lemma}\label{Lemma 4.3}
Let $S$ be a semigroup and let $A\subseteq S$. Let $p$ be an idempotent in $\beta S$ with $A\in p$. There is a $FP$-tree $\mathsf{T}$ in $A$ such that for each $f\in \mathsf{T},\, B_f\in p$.
\end{lemma}

\begin{lemma}\label{Lemma 4.4}
Any $FP$-tree is a $*$-tree.
\end{lemma}

When we say $\langle C_{F}\rangle_{F\in J}$ is a "downward directed family", we mean $J$ is a directed set and for $F,G\in J$ if $F\leq G$, then  $C_{G}\subseteq C_{F}$. 

\begin{theorem}\label{Theorem 4.5}
Let $S$ be a semigroup with a given net $\mathcal{F}$ in $\mathcal{P}_f(S)$, and let $A\subseteq S$. Let either of the following two conditions hold:
\begin{enumerate}
\item There is some $b\in\mathbb{N}$ such that $S$ is $b$-weakly right cancellative.
\item There is some $b\in\mathbb{N}$ such that $S$ is $b$-weakly left cancellative, and $\mathcal{F}$ satisfies the $(*)$ condition.
\end{enumerate}
\renewcommand{\labelenumi}{(\alph{enumi})}\renewcommand{\labelenumii}{(\roman{enumii})}
Then statements $(a),(b),(c)$ and $(d)$ are equivalent and implied by statement $(e)$. If $S$ is countable then all the statements are equivalent.
\begin{enumerate}
\item $A$ is a $D$-set with respect to $\mathcal{F}$.
\item There is a $FP$-tree $\mathsf{T}$ in $A$ such that for each $F\in \mathcal{P}_f(\mathsf{T}),\,\,d_{\mathcal{F}}^*\big(\bigcap_{f\in F}B_f\big)>0$.
\item There is a $*$-tree $\mathsf{T}$ in $A$ such that for each $F\in \mathcal{P}_f(\mathsf{T}),\,d_{\mathcal{F}}^*\big(\bigcap_{f\in F}B_f\big)>0$.
\item There is a downward directed family $\langle C_F\rangle_{F\in J}$ of subsets of $A$ such that 
\begin{enumerate}
\item for each ${F\in J}$ and each $s\in C_F$ there exists $G\in J$ with $C_G  \subseteq s^{-1}C_F$ and
\item for each $F\in J,\; d_{\mathcal{F}}^*(C_F)>0$. 
\end{enumerate}
\item There is a decreasing sequence $\langle C_n\rangle_{n\in\mathbb{N}}$ of subsets of $A$ such that 
\begin{enumerate}
\item for each $n\in \mathbb{N}$ and each $s\in C_n$, there exists $m\in \mathbb{N}$ with $C_m\subseteq s^{-1}C_n$ and
\item for each $n\in \mathbb{N},\; d_{\mathcal{F}}^*(C_n)>0$.
\end{enumerate}
\end{enumerate}
\end{theorem}

\begin{proof}
$\mathit{(a)}\Rightarrow\mathit{(b)}$: Pick an idempotent $p\in D_{\mathcal{F}}^{*}$ such that $A\in p$. Pick a $FP$-tree $\mathsf{T}$ in $A$ with $B_{f}\in p$ for all $f\in \mathsf{T}$ (Lemma \ref{Lemma 4.3}). Given $F\in\mathcal{P}_{f}(\mathsf{T})$, one has that for each $f\in F,\;B_{f}(\mathsf{T})\in p$, so $\bigcap_{f\in F}B_{f}(\mathsf{T})\in p$. Since $p\in D_{\mathcal{F}}^{*}$, we have $d_{\mathcal{F}}^{*}\big(\bigcap_{f\in F}B_{f}(\mathsf{T})\big)>0$.\par
$\mathit{(b)}\Rightarrow\mathit{(c)}$: This follows directly from Lemma \ref{Lemma 4.4}.\par
$\mathit{(c)}\Rightarrow\mathit{(d)}$: Let $\mathsf{T}$ be given as guaranteed by $\mathit{(c)}$. Let $J=\mathcal{P}_{f}(\mathsf{T})$ and for each $F\in J,\,C_{F}=\bigcap_{f\in F}B_{f}(\mathsf{T})$. Then $\langle C_F\rangle_{F\in J}$ is a downward directed family and $d_{\mathcal{F}}^{*}(C_{F})>0$ for each $F\in J$. Let $F\in J$ and let $s\in C_{F}$. Let $G=\{f^{\frown} s:f\in F\}$. Now for each $f\in F$ we have $B_{f^{\frown} s}\subseteq s^{-1}B_{f}$. So $C_{G}\subseteq s^{-1}C_{F}$.\par
$\mathit{(d)}\Rightarrow\mathit{(a)}$: First we claim that $\big\{C_F:F\in J\big\}$ satisfies the finite intersection property, i.e. for each $J_0\in\mathcal{P}_f(J)$, $\bigcap_{F\in J_0}C_F\neq\emptyset$. Note that as $J$ is a directed set, so each pair of elements has an upper bound. This implies that each finite collection of elements of $J$ also has an upper bound. For a given $J_0\in\mathcal{P}_f(J)$, let $G\in J$ be a common upper bound for all $F\in J_0$, i.e. $F\leq G$ for each $F\in J_0$. As $\langle C_F\rangle_{F\in J}$ is a downward directed family, so we have $C_G\subseteq C_F$ for each $F\in J_0$ and intersecting over $J_0$, we get $C_G\subseteq\bigcap_{F\in J_0}C_F$. But as $d_{\mathcal{F}}^*(C_G)>0$, hence in particulur $C_G\neq\emptyset$, and we get $\bigcap_{F\in J_0}C_F\neq\emptyset$, as claimed. Now by \cite[Theorem 4.20]{hs}, $M=\bigcap_{F\in J}\overline{C_{F}}$ is a compact subsemigroup of $\beta S$. So, it suffices to show $M\cap D_{\mathcal{F}}^{*}\neq\emptyset $, as this ensures the existence of an idempotent in $M\cap D_{\mathcal{F}}^{*}$. Then as $M\subseteq\overline{A}$, it follows from the definition that $A$ is a $D$-set. Now given $F\in J$, since $d_{\mathcal{F}}^{*}(C_{F})>0$, by Lemma \cite[Theore 3.11]{hs} $\overline{C_{F}}\cap D_{\mathcal{F}}^{*}\neq\emptyset$. Since $\langle C_{F}\rangle_{F\in J}$ is a downward directed family, so $M\cap D_{\mathcal{F}}^{*}\neq\emptyset$.\par
$\mathit{(e)}\Rightarrow\mathit{(d)}$: This follows trivially.\par
Assume now that $S$ is countable. We will show that $\mathit{(c)}$ implies $\mathit{(e)}$. Let $\mathsf{T}$ be guaranteed by $\mathit{(c)}$. So $\mathsf{T}$ is countable. Enumerate $\mathsf{T}$ as $\{f_{n}:n\in\mathbb{N}\}$. For each $n\in\mathbb{N}$, let $C_{n}=\bigcap_{k=1}^{n}B_{f_{k}}(\mathsf{T})$. Then for each $n\in \mathbb{N}$, $d_{\mathcal{F}}^*(C_n)>0$. Let $n\in\mathbb{N}$ and let $s\in C_{n}$. Pick $m\in\mathbb{N}$ such that $\{f_{k}^{\frown} s:k\in\{1,2,...,m\}\}\subseteq\{f_{1},f_{2},...,f_{m}\}$. Then $C_{m}\subseteq s^{-1}C_{n}$.
\end{proof}

\section{Product of D-sets}\label{Section 5}

In this section we shall determine when the Cartesian products of $D$-sets in arbitrary semigroups is a $D$-set. To begin with, suppose we have two semigroups $S$ and $W$ with given nets $\mathcal{E}=\langle E_i \rangle_{i\in I}$ and $\mathcal{F}=\langle F_j\rangle_{j\in J}$ respectively in $\mathcal{P}_f(S)$ and $\mathcal{P}_f(W)$. Consider the direct product semigroup $S\times W$. Order $I\times J$ by agreeing that $(i_1,j_1)\geq (i_0,j_0)$ provided $i_1\geq i_0$ and $j_1\geq j_0$. Define the product net $\mathcal{E}*\mathcal{F}=\langle E_i\times F_j\rangle_{(i,j)\in I\times J}$.

\begin{lemma}\label{Lemma 5.1}
Let $S$ and $W$ be two semigroups with given nets $\mathcal{E}=\langle E_i \rangle_{i\in I}$ and $\mathcal{F}=\langle F_j\rangle_{j\in J}$ respectively in $\mathcal{P}_f(S)$ and $\mathcal{P}_f(W)$. Let each of $S$ and $W$ have an identity element. Then for $A\subseteq S$ and $B\subseteq W,\,d_{\mathcal{E}*\mathcal{F}}^{*}(A\times B)=d_{\mathcal{E}}^{*}(A)\cdot d_{\mathcal{F}}^{*}(B)$. 
\end{lemma}

\begin{proof}
Let $\alpha_1=d_{\mathcal{E}}^*(A)$ and let $\beta_1=d_{\mathcal{F}}^*(B)$. Suppose first that $d_{\mathcal{E}*\mathcal{F}}^*(A\times B)>\alpha_1\cdot \beta_1$. Pick $\epsilon >0$ and $\gamma$ such that $d_{\mathcal{E}*\mathcal{F}}^*(A\times B) > \gamma > (\alpha_1 + \epsilon) \cdot (\beta_1 + \epsilon)$. Pick $i_0\in I$ such that for all $i\geq i_0$ and all $s\in S$, $|A \cap (E_i \cdot s)| < (\alpha_1 + \epsilon) \cdot |E_i|$ and pick $j_0\in J$ such that for all $j\geq j_0$ and all $w \in W$, $|B \cap (F_j \cdot w)| < (\beta_1 + \epsilon) \cdot |F_j|$. Pick $(i,j) \in I \times J$ and $(s,w) \in S \times W$ such that $\big|(A \times B) \cap \big((E_i \times F_j) \cdot (s,w)\big)\big| \geq \gamma \cdot |E_i \times F_j|$. Then $\gamma \cdot |E_i| \cdot |F_j| = \gamma \cdot |E_i \times F_j| \leq \big|(A \times B) \cap \big((E_i \times F_j) \cdot (s,w)\big)\big| = |A \cap (E_i \cdot s)| \cdot |B \cap (F_j \cdot w)| < (\alpha_1 + \epsilon) \cdot |E_i| \cdot (\beta_1 + \epsilon) \cdot |F_j|$, and we have that $\gamma < (\alpha_1 + \epsilon) \cdot (\beta_1 + \epsilon)$, a contradiction. \par
Now suppose that $d_{\mathcal{E}* \mathcal{F}}^*(A\times B) < \alpha_1 \cdot \beta_1$. Pick $\epsilon > 0$ and $\gamma$ such that $d_{\mathcal{E}* \mathcal{F}}^*(A\times B) < \gamma < (\alpha_1 - \epsilon) \cdot (\beta_1 - \epsilon)$. Pick $(i_0,j_0)\in I\times J$ such that for all $(i,j) \geq (i_0,j_0)$ and all $(s,w)\in S\times W$, $\big|(A \times B) \cap \big((E_i \times F_j) \cdot (s,w)\big)\big| < \gamma \cdot |E_i \times F_j|$. Pick $i \geq i_0$ and $s\in S$ such that $|A \cap (E_i \cdot s)| \geq (\alpha_1 - \epsilon) \cdot |E_i|$. Also pick $j \geq j_0$ and $w \in W$ such that $|B \cap (F_j \cdot w)|\geq (\beta_1 - \epsilon) \cdot |F_j|$. Then $\gamma \cdot |E_i| \cdot |F_j| = \gamma \cdot |E_i \times F_j| > \big|(A \times B) \cap \big((E_i \times F_j) \cdot (s,w)\big)\big| = |A \cap (E_i \cdot s)| \cdot |B \cap (F_j \cdot w)| \geq (\alpha_1 - \epsilon) \cdot (\beta_1 - \epsilon) \cdot |E_i| \cdot |F_j|$, and we have that $\gamma > (\alpha_1 - \epsilon) \cdot (\beta_1 - \epsilon)$, a contradiction. Thus, we have $d_{\mathcal{E}*\mathcal{F}}^{*}(A\times B)=d_{\mathcal{E}}^{*}(A)\cdot d_{\mathcal{F}}^{*}(B)$. 
\end{proof}

\begin{remark}
The conclusion of Lemma \ref{Lemma 5.1} may not be true if the semigroups $S$ and $W$ do not have identity elements. For example consider the semigroup $(\mathbb{N},+)$. Consider the two nets $\mathcal{E} = \langle E_n \rangle_{n\in\mathbb{N}}$ and $\mathcal{F} = \langle F_n \rangle_{n\in\mathbb{N}}$ where $E_n = \{1\}$ and $F_n = \{2\}$ for all $n\in\mathbb{N}$. Then $d_{\mathcal{E}}^*\big(\{2\}\big)=d_{\mathcal{F}}^*\big(\{2\}\big)=1$, but $d_{\mathcal{E} * \mathcal{F}}^*\big(\{(2,2)\}\big)=0$.
\end{remark}

\begin{lemma}\label{Lemma 5.3}
Let $S$ and $W$ be two semigroups. Let $\tilde{\iota}:\beta(S\times W)\rightarrow\beta S\times\beta W$ be the continuous extension of the inclusion map $\iota : S\times W \rightarrow \beta S \times \beta W$. Then for $p\in\beta S$ and $q\in\beta W$
$$\bigcap_{A\in p,B\in q}\overline{A\times B}\subseteq\tilde{\iota}^{-1}(\{(p,q)\})$$
\end{lemma}

\begin{proof}
The subset $\bigcap_{A\in p,B\in q}\overline{A\times B}$ of $\beta (S\times W)$ is nonempty as the collection $\big\{\overline{A\times B}:A\in p,B\in q\big\}$ satisfies the finite intersection property. Note that $\tilde{\iota}\big(\bigcap_{A\in p,B\in q}\overline{A\times B}\big)\subseteq\bigcap_{A\in p,B\in q}\tilde{\iota}\big(\overline{A\times B}\big)$ and by the continuity of $\tilde{\iota}$ we have $\bigcap_{A\in p,B\in q}\tilde{\iota}\big(\overline{A\times B}\big)\subseteq\bigcap_{A\in p,B\in q}\overline{\tilde{\iota}(A\times B)}$. But $\overline{\tilde{\iota}(A\times B)}=\overline{A}\times\overline{B}$ and finally we have $\tilde{\iota}\big(\bigcap_{A\in p,B\in q}\overline{A\times B}\big)\subseteq \bigcap_{A\in p, B\in q}\overline{A}\times \overline{B}= \{(p,q)\}$, i.e. $\bigcap_{A\in p,B\in q}\overline{A\times B}\subseteq\tilde{\iota}^{-1}(\{(p,q)\})$.
\end{proof}

\begin{lemma}\label{Lemma 5.4}
Let $S$ and $W$ be two semigroups with given nets $\mathcal{E}$ and $\mathcal{F}$ respectively in $\mathcal{P}_f(S)$ and $\mathcal{P}_f(W)$. Let each of $S$ and $W$ have an identity element. Let $\tilde{\iota}:\beta(S\times W)\rightarrow\beta S\times\beta W$ be the continuous extension of the inclusion map $\iota : S\times W \rightarrow \beta S \times \beta W$. Then for any $p$ in $D_{\mathcal{E}}^{*}$ and $q$ in $D_{\mathcal{F}}^{*},\,\tilde{\iota}^{-1}(\{(p,q)\})\cap D_{\mathcal{E}*\mathcal{F}}^{*}\neq\emptyset $. 
\end{lemma}

\begin{proof}
By Lemma \ref{Lemma 5.3}, it suffices to show that $\bigcap_{A\in p,B\in q}\overline{A\times B}\cap D_{\mathcal{E}* \mathcal{F}}^*\neq\emptyset$. This will follow if we show that the collection of closed sets $\big\{\overline{A\times B}\cap D_{\mathcal{E}* \mathcal{F}}^*:A\in p,B\in q\big\}$ has finite intersection property. But as each of the ultrafilters $p$ and $q$ is closed under any finite intersection of its elements, so it suffices to show that for each $A\in p$ and $B\in q$, we have $\overline{A\times B}\cap D_{\mathcal{E}* \mathcal{F}}^*\neq \emptyset$. To prove this, consider the collection $\mathcal{R}=\{C\subseteq S\times W:d_{\mathcal{E}*\mathcal{F}}^{*}(C)>0\}$. Note that $\emptyset\notin\mathcal{R},\mathcal{R}^{\uparrow}=\mathcal{R}\text{ and }\mathcal{R}$ is partition regular. Now for each $A\in p$ and $B\in q$, we have $A\times B\in\mathcal{R}$ and therefore \cite[Theorem 3.11]{hs} provides us an ultrafilter $r$ in $\beta(S\times W)$ such that $A\times B\in r\subseteq\mathcal{R}$, i.e. $r\in \overline{A\times B}\cap D_{\mathcal{E}*\mathcal{F}}^{*}$.
\end{proof}

\begin{lemma}\label{Lemma 5.5}
Let $S$ and $Y$ be two semigroups and $\pi:S\rightarrow Y$ be an onto homomorphism. Then $\pi$ extends continuously to an onto
homomorphism $\tilde{\pi}:\beta S\rightarrow\beta Y$ and for any $q\in\beta Y,\,\bigcap_{B\in q}\overline{\pi^{-1}(B)}\subseteq\tilde{\pi}^{-1}(\{q\})$. 
\end{lemma}

\begin{proof}
By \cite[Corollary
4.22 and Exercise 3.4.1]{hs}, the continuous extension $\tilde{\pi}:\beta S \rightarrow \beta Y$ is an onto homomorphism. Now pick an ultrafilter $q\in \beta Y$ and note that $\{q\}=\bigcap_{B\in q}\overline{B}$. Hence $\tilde{\pi}^{-1}(\{q\})=\tilde{\pi}^{-1}\big(\bigcap_{B\in q}\overline{B}\big)=\bigcap_{B\in q}\tilde{\pi}^{-1}\big(\overline{B}\big)\supseteq \bigcap_{B\in q}\overline{\pi^{-1}(B)}$ where the last inclusion follows from the continuity of $\tilde{\pi}$.
\end{proof}

We omit the proof of the following lemma, since it is essentially similar to the proof of Lemma \ref{Lemma 5.4}.

\begin{lemma}\label{Lemma 5.6}
Let $S$ and $Y$ be two semigroups with given nets $\mathcal{E}_1$ and $\mathcal{E}_2$ respectively in $\mathcal{P}_f(S)$ and $\mathcal{P}_f(Y)$. Let each of $S$ and $Y$ have an identity element. Let $\pi:S\rightarrow Y$ be an onto homomorphism such that $d_{\mathcal{E}_{1}}^{*}(\pi^{-1}(B))=d_{\mathcal{E}_{2}}^{*}(B)$ for all $B\subseteq Y$. Let $\tilde{\pi}:\beta S \rightarrow \beta Y$ be the continuous extension of $\pi$. Then for any $q$ in $D_{\mathcal{E}_{2}}^{*},\,\tilde{\pi}^{-1}(q)\cap D_{\mathcal{E}_{1}}^{*}\neq\emptyset $. 
\end{lemma}

\begin{proposition}\label{Proposition 5.7}
Let $S$ and $Y$ be two semigroups with given nets $\mathcal{E}_1$ and $\mathcal{E}_2$ respectively in $\mathcal{P}_f(S)$ and $\mathcal{P}_f(Y)$. Let each of $S$ and $Y$ have an identity element. Let either of the following two conditions hold:
\begin{enumerate}
\item There is some $b\in\mathbb{N}$ such that $S$ is $b$-weakly right cancellative.
\item There is some $b\in\mathbb{N}$ such that $S$ is $b$-weakly left cancellative, and $\mathcal{E}_1$ satisfies the $(*)$ condition.
\end{enumerate}
Let $\pi:S\rightarrow Y$ be an onto homomorphism such that $d_{\mathcal{E}_{1}}^{*}(\pi^{-1}(B))=d_{\mathcal{E}_{2}}^{*}(B)$
for all $B\subseteq Y$. Then the induced homomorphism $\tilde{\pi}:\beta S\rightarrow\beta Y$ sends $D_{\mathcal{E}_{1}}^{*}$ onto $D_{\mathcal{E}_{2}}^{*}$. Also, $\tilde{\pi}$ sends the idempotents of $D_{\mathcal{E}_{1}}^{*}$ onto the idempotents of $D_{\mathcal{E}_{2}}^{*}$. In particular $\pi$ sends the $D$-sets (with respect to $\mathcal{E}_1$) in $S$ onto the $D$-sets (with respect to $\mathcal{E}_2$) in $Y$. 
\end{proposition}

\begin{proof}
Let $p\in D_{\mathcal{E}_{1}}^{*}$ be given. Then by \cite[Lemma 3.30]{hs}, $\tilde{\pi}(p)=\{C\subseteq Y:\pi^{-1}(C)\in p\}$. So for any $E$ in $\tilde{\pi}(p),\,d_{\mathcal{E}_{2}}^{*}(E)=d_{\mathcal{E}_{1}}^{*}(\pi^{-1}(E))>0$, and thus $\tilde{\pi}(p)$ is in $D_{\mathcal{E}_{2}}^{*}$. This together with Lemma \ref{Lemma 5.6}  now shows that $\tilde{\pi}$ sends $D_{\mathcal{E}_{1}}^{*}$ onto $D_{\mathcal{E}_{2}}^{*}$.\par
Moreover, since $\tilde{\pi}$ is a homomorphism, so it will send the idempotents of $D_{\mathcal{E}_1}^*$ into the idempotents of $D_{\mathcal{E}_2}^*$. Also, for an idempotent $q$ in $D_{\mathcal{E}_2}^*$, $\tilde{\pi}^{-1}(q)\cap D_{\mathcal{E}_{1}}^{*}\neq\emptyset$ by Lemma \ref{Lemma 5.6} and hence $\tilde{\pi}^{-1}(q)\cap D_{\mathcal{E}_{1}}^{*}$ is a closed subsemigroup of $\beta S$. Now by Ellis' theorem, pick an idempotent $p$ in $\tilde{\pi}^{-1}(q)\cap D_{\mathcal{E}_{1}}^{*}$ and we have $\tilde{\pi}(p)=q$ where $p$ is an idempotent in $D_{\mathcal{E}_1}^*$. Thus, $\tilde{\pi}$ sends the idempotents of $D_{\mathcal{E}_1}^*$ onto the idempotents of $D_{\mathcal{E}_2}^*$, i.e. in particular $\pi$ sends the $D$-sets (with respect to $\mathcal{E}_1$) in $S$ onto the $D$-sets (with respect to $\mathcal{E}_2$) in $Y$. 
\end{proof}

\begin{theorem}\label{Theorem 5.8}
Let $b\in\mathbb{N}$ and assume that $S$ and $W$ are two $b$-weakly right cancellative semigroups with given nets $\mathcal{E}$ and $\mathcal{F}$ respectively in $\mathcal{P}_f(S)$ and $\mathcal{P}_f(W)$. Let each of $S$ and $W$ have an identity element. Let $A$ be a subset of $S$ and $B$ be a subset of $W$. Then $A\times B$ is a $D$-set (with respect to $\mathcal{E}*\mathcal{F}$) in $S\times W$ if and only if both $A$ and $B$ are $D$-sets (with respect to $\mathcal{E}$ and $\mathcal{F}$ respectively) in $S$ and $W$ respectively. 
\end{theorem}

\begin{proof}
Let $A$ and $B$ be two given $D$-sets with respect to $\mathcal{E}$ and $\mathcal{F}$ respectively. Pick idempotents $p$ in $D_{\mathcal{E}}^{*}$ and $q$ in $D_{\mathcal{F}}^{*}$ such that $A$ is in $p$ and $B$ is in $q$. Consider the continuous extension $\tilde{\iota}:\beta (S\times W)\rightarrow \beta S \times \beta W$ of the inclusion map $\iota : S\times W \rightarrow \beta S \times \beta W$. By \cite[Corollary 4.22]{hs}, the map $\tilde{\iota}$ is a homomorphism. Hence by Lemma \ref{Lemma 5.4}, $\tilde{\iota}^{-1}(\{(p,q)\})\cap D_{\mathcal{E}*\mathcal{F}}^{*}$ is a compact subsemigroup of $\beta(S\times T)$. So by Ellis' theorem, pick an idempotent $r\in\tilde{\iota}^{-1}(\{(p,q)\})\cap D_{\mathcal{E}*\mathcal{F}}^{*}$. Since $\tilde{\iota}(r)=(p,q)$ and $\overline{A} \times \overline{B}$ is a neighborhood of $(p,q)$, there is some $C\in r$ such that $\tilde{\iota}(C) \subseteq \overline{A} \times \overline{B}$ and so $C \subseteq A \times B$. This implies $A \times B \in r$ and therefore, $A\times B$ is a $D$-set (with respect to $\mathcal{E}*\mathcal{F}$) in $S\times W$.\par
Conversely, let $A\times B$ be a $D$-set (with respect to $\mathcal{E}*\mathcal{F}$) in $S\times W$. Consider the projection homomorphism $\pi_{1}:S\times W\rightarrow S$ onto $S$. Then by Lemma \ref{Lemma 5.1}, $d_{\mathcal{E}*\mathcal{F}}^*\big(\pi_1^{-1}(E)\big)=d_{\mathcal{E}*\mathcal{F}}^*(E\times W)=d_{\mathcal{E}}^*(E)\cdot d_{\mathcal{F}}^*(W)=d_{\mathcal{E}}^*(E)$ for each subset $E$ of $S$. Therefore by Proposition \ref{Proposition 5.7}, $A=\pi_{1}(A\times B)$ is a $D$-set with respect to $\mathcal{E}$. Similarly, we can prove that $B$ is a $D$-set with respect to $\mathcal{F}$.
\end{proof}

\begin{remark}\label{Remark 5.9}
Inductively we can prove that for any finite collection of $b$-weakly right cancellative semigroups $\{S_{\imath}:\imath=1,2,...,n\}$, each of which contains an identity element, a subset $\bigtimes_{\imath=1}^{n}A_{\imath}$ \big(of $\bigtimes_{\imath=1}^{n}S_{\imath}$\big) is a $D$-set in $\bigtimes_{\imath=1}^{n}S_{\imath}$ if and only if each $A_{\imath}$ is a $D$-set in $S_{\imath}$. 
\end{remark}

Now we will consider an infinite product of $D$-sets. Let $\{S_\imath:\imath\in \mathcal{I}\}$ be a given collection of semigroups, each with an identity $1$, and for $\imath\in \mathcal{I}$, let $\mathcal{F}_\imath=\langle F_{\imath,j}\rangle_{j\in J_\imath}$ be a net in $\mathcal{P}_f(S_\imath)$. It would be natural to use the net $\big\langle\bigtimes_{\imath\in \mathcal{I}}F_{\imath,j(\imath)}\big\rangle_{\vec{j}\in\bigtimes_{\imath\in \mathcal{I}}J_\imath}$ to define density on $\bigtimes_{\imath\in \mathcal{I}}S_\imath$, but given $\vec{j}\in\bigtimes_{\imath\in \mathcal{I}}J_\imath,\,\bigtimes_{\imath\in \mathcal{I}}F_{\imath,j(\imath)}$ is not in $\mathcal{P}_f\big(\bigtimes_{\imath\in \mathcal{I}}S_\imath\big)$ unless $F_{\imath,j(\imath)}$ is a singleton for all but finitely many values of $\imath$. So we will construct a new product net as follows.\par
For $\imath\in \mathcal{I}$, let $J_\imath ^{'}=J_\imath \cup \{1\}$ where $1$ is a point not in $J_\imath$ and define $F_{\imath,1}=\{1\}$, specifying $1\leq j$ for all $j\in J_\imath$. Let $\mathfrak{J}=\big\{\vec{j}\in\bigtimes_{\imath\in \mathcal{I}}J_\imath ^{'}:\{\imath\in \mathcal{I}:j(\imath)\neq 1\} \text{ is finite}\big\}$. For $\vec{j},\vec{k}\in \mathfrak{J}$, agree that $\vec{j}\leq \vec{k}$ provided that for each $\imath\in \mathcal{I},\,j(\imath)\leq k(\imath)$. Then the product net determined by $\{\mathcal{F}_\imath:\imath\in \mathcal{I}\}$ is the net $\mathcal{F}=\big\langle\bigtimes_{\imath\in \mathcal{I}}F_{\imath,j(\imath)}\big\rangle_{\vec{j}\in \mathfrak{J}}$.\par
When we refer to the product net determined by $\{\mathcal{F}_\imath:\imath\in \mathcal{I}\}$, we will assume that we have $J_\imath ^{'}$ for $\imath\in \mathcal{I}$ and $\mathfrak{J}$ as in the above paragraph. Note that, when $\mathcal{I}$ is finite, then the above product net is essentially same with the usual product of nets. Therefore, the corresponding density is the usual density with respect to the product net. The following proposition differentiates an infinite product net from a finite one, as the conclusion holds only for infinite products.

\begin{proposition}
Let $\{S_\imath:\imath\in\mathcal{I}\}$ be an infinite collection of semigroups with given nets $\mathcal{F}_{\imath}=\langle F_{\imath,j}\rangle_{j\in J_{\imath}}$ in $\mathcal{P}_f(S_\imath)$ for each $\imath\in\mathcal{I}$ and assume that each $S_\imath$ has an identity element $1$. Let $\mathcal{F}$ be the product net determined by $\{\mathcal{F}_\imath:\imath\in\mathcal{I}\}$. Then $d_{\mathcal{F}}^*\Big(\big(\bigtimes_{\imath\in\mathcal{I}}S_\imath \big) \setminus \big(\bigtimes_{\imath\in\mathcal{I}}(S_\imath \setminus \{1\})\big)\Big)=1$.
\end{proposition}

\begin{proof}
For $\vec{j}\in \mathfrak{J}$, let $F_{\vec{j}}=\bigtimes_{\imath\in\mathcal{I}}F_{\imath,j(\imath)}$. We will show that for each $\vec{k}\in \mathfrak{J}$ there exists $\vec{j}\in \mathfrak{J}$ and $\vec{s}\in\bigtimes_{\imath\in\mathcal{I}}S_\imath$ such that $\vec{j}\geq \vec{k}$ and 
$$\big|\big(\big({\bigtimes}_{\imath\in\mathcal{I}}S_\imath \big) \setminus \big({\bigtimes}_{\imath\in\mathcal{I}}(S_\imath \setminus \{1\})\big)\big)\cap F_{\vec{j}}\cdot \vec{s}\big|\geq |F_{\vec{j}}|.$$
So let $\vec{k}\in \mathfrak{J}$ be given and let $\vec{j}=\vec{k}$. Let $\vec{s}=\langle s_\imath \rangle_{\imath\in\mathcal{I}}$ be the vector such that each $s_\imath=1$. Note that $F_{\vec{j}}\cdot \vec{s}=F_{\vec{j}}$. Pick $\imath_0 \in \mathcal{I}$ such that $j(\imath_0)=1$. Now 

\begin{equation*}
\begin{split}
F_{\vec{j}} &=\big({\bigtimes}_{\imath\in\mathcal{I}}S_{\imath}\big)\cap F_{\vec{j}}\\
&=\big(\big(\big({\bigtimes}_{\imath\in\mathcal{I}}S_{\imath} \big) \setminus \big( {\bigtimes}_{\imath\in\mathcal{I}}(S_{\imath} \setminus \{1\})\big)\big)\cap F_{\vec{j}}\big)\cup \big(\big({\bigtimes}_{\imath\in\mathcal{I}}(S_{\imath} \setminus \{1\})\big)\cap F_{\vec{j}}\big).
\end{split}
\end{equation*}

Since $j(\imath_0)=1$, $F_{\imath_0,j(\imath_0)}=\{1\}$ and therefore $(S_{\imath_0} \setminus \{1\})\cap F_{\imath_0,j(\imath_0)}=\emptyset$ so that $\big({\bigtimes}_{\imath\in\mathcal{I}}(S_{\imath} \setminus \{1\})\big)\cap F_{\vec{j}}=\emptyset$. Therefore $\big|\big(\big({\bigtimes}_{\imath\in\mathcal{I}}S_{\imath} \big) \setminus \big( {\bigtimes}_{\imath\in\mathcal{I}}(S_{\imath} \setminus \{1\})\big)\big)\cap F_{\vec{j}}\big|=|F_{\vec{j}}|$ as required.

\end{proof}

The next lemma is a generalization of Lemma \ref{Lemma 5.1} for an infinite collection of semigroups. Note that, for an infinite collection $\{\alpha_\imath:\imath\in\mathcal{I}\}\subseteq [0,1]$, we define their product as $\prod_{\imath\in\mathcal{I}}\alpha_\imath:=\inf\{\prod_{\imath\in G}\alpha_\imath:G\in\mathcal{P}_f(\mathcal{I})\}$.\\

\begin{lemma}\label{Lemma 5.11}
Let $\{S_\imath:\imath\in\mathcal{I}\}$ be an infinite collection of semigroups with given nets $\mathcal{F}_{\imath}=\langle F_{\imath,j}\rangle_{j\in J_{\imath}}$ in $\mathcal{P}_f(S_\imath)$ for each $\imath\in\mathcal{I}$ and assume that each $S_\imath$ has an identity element $1$. Let $\mathcal{F}$ be the product net determined by $\{\mathcal{F}_\imath:\imath\in\mathcal{I}\}$. Let $A_{\imath}\subseteq S_{\imath}$ for each $\imath\in\mathcal{I}$. Then $d_{\mathcal{F}}^{*}\big(\bigtimes_{\imath\in\mathcal{I}}A_{\imath}\big)=\prod_{\imath\in\mathcal{I}}d_{\mathcal{F}_{\imath}}^{*}(A_{\imath})$. 
\end{lemma}

\begin{proof}
For each $\imath\in\mathcal{I}$, let $\delta_\imath=d_{\mathcal{F}_\imath}^*(A_\imath)$. We may presume that each $A_\imath\neq\emptyset$. Suppose first that $d_{\mathcal{F}}^*\big(\bigtimes_{\imath\in\mathcal{I}}A_\imath\big)>\prod_{\imath\in\mathcal{I}}\delta_\imath$ and pick $\gamma$ such that $d_{\mathcal{F}}^*\big(\bigtimes_{\imath\in\mathcal{I}}A_\imath\big) > \gamma > \prod_{\imath\in\mathcal{I}}\delta_\imath$. Pick $G\in \mathcal{P}_f(\mathcal{I})$ such that $\prod_{\imath\in G}\delta_\imath < \gamma$ and pick $\epsilon >0$ such that $\prod_{\imath\in G}(\delta_\imath +\epsilon) < \gamma$.\par
For $\imath\in G$, pick $t(\imath)\in J_\imath$ such that whenever $j\in J_\imath$ with $j\geq t(\imath)$ and $s\in S_\imath$, one has $|A_\imath \cap (F_{\imath,j}\cdot s)| < (\delta_\imath + \epsilon)\cdot |F_{\imath,j}|$. We can presume that each $t(\imath)\neq 1$. Define $\vec{j}\in \mathfrak{J}$ by $j(\imath)=t(\imath)$ if $\imath\in G$ and $j(\imath)=1$ if $\imath\in\mathcal{I}\setminus G$. Pick $\vec{k}\in \mathfrak{J}$ such that $\vec{k}\geq \vec{j}$ and $\vec{s}=\langle s_\imath \rangle_{\imath\in\mathcal{I}}\in \bigtimes_{\imath\in\mathcal{I}}S_\imath$ such that $\big|\big(\bigtimes_{\imath\in\mathcal{I}}A_\imath \big) \cap \big(\big(\bigtimes_{\imath\in\mathcal{I}}F_{\imath,k(\imath)}\big)\cdot \vec{s}\big)\big| > \gamma \cdot \big|\bigtimes_{\imath\in\mathcal{I}}F_{\imath,k(\imath)}\big|$.\par
Let $E=\{\imath\in\mathcal{I}:k(\imath)\neq 1\}$. Then $G\subseteq E$. (If $G=E$, delete the references to $E\setminus G$ in the computation that follows.) Note that if $\imath\in\mathcal{I}\setminus E$, then $F_{\imath,k(\imath)}\cdot s_\imath=\{s_\imath\}$ so $\big|A_\imath \cap \big(F_{\imath,k(\imath)}\cdot s_\imath\big)\big|\leq 1$. We have then 

\begin{equation*}
\begin{split}
\gamma \cdot \big|{\bigtimes}_{\imath\in\mathcal{I}}F_{\imath,k(\imath)}\big| &< \big|\big({\bigtimes}_{\imath\in\mathcal{I}}A_\imath \big) \cap \big(\big({\bigtimes}_{\imath\in\mathcal{I}}F_{\imath,k(\imath)}\big)\cdot \vec{s}\big)\big|\\
&\leq \big|{\bigtimes}_{\imath\in E}\big(A_\imath\cap \big(F_{\imath,k(\imath)}\cdot s_\imath\big)\big)\big|\\
&=\big|{\bigtimes}_{\imath\in E\setminus G}\big(A_\imath\cap\big(F_{\imath,k(\imath)}\cdot s_\imath\big)\big)\big|\cdot\big|{\bigtimes}_{\imath\in G}\big(A_\imath\cap \big(F_{\imath,k(\imath)}\cdot s_\imath \big)\big)\big|\\
&\leq {\prod}_{\imath\in E\setminus G}\big|F_{\imath,k(\imath)}\big|\cdot {\prod}_{\imath\in G}\big((\delta_\imath + \epsilon )\cdot \big|F_{\imath,k(\imath)}\big|\big)\\
&<\gamma\cdot{\prod}_{\imath\in\mathcal{I}}\big|F_{\imath,k(\imath)}\big|,
\end{split}
\end{equation*}
a contradiction.\par
Now suppose that $d_{\mathcal{F}}^*\big(\bigtimes_{\imath\in\mathcal{I}}A_\imath\big) < \prod_{\imath\in\mathcal{I}}\delta_\imath$ and pick $\gamma$ such that $d_{\mathcal{F}}^*\big(\bigtimes_{\imath\in\mathcal{I}}A_\imath\big) < \gamma < \prod_{\imath\in\mathcal{I}}\delta_\imath$. Pick $\vec{j}\in \mathfrak{J}$ such that whenever $\vec{k}\in \mathfrak{J}$ with $\vec{k}\geq \vec{j}$ and $\vec{s}\in\bigtimes_{\imath\in\mathcal{I}}S_\imath$, one has $\big|\big(\bigtimes_{\imath\in\mathcal{I}}A_\imath \big)\cap \big(\big(\bigtimes_{\imath\in\mathcal{I}}F_{\imath,k(\imath)}\big)\cdot \vec{s} \big)\big| < \gamma\cdot \big|\bigtimes_{\imath\in\mathcal{I}}F_{\imath,k(\imath)}\big|$. Let $G=\{\imath\in\mathcal{I}:j(\imath)\neq 1\}$. Pick $\epsilon > 0$ such that $\prod_{\imath\in G}(\delta_\imath - \epsilon) > \gamma$. For $\imath\in G$, pick $t(\imath)\in J_{\imath}$ such that $t(\imath)\geq j(\imath)$ and $s_\imath\in S_\imath$ such that $\big|A_\imath\cap\big(F_{\imath,t(\imath)}\cdot s_\imath\big)\big|>(\delta_\imath - \epsilon)\cdot \big|F_{\imath,t(\imath)}\big|$. Define $\vec{k}\in \mathfrak{J}$ by $k(\imath)=t(\imath)$ if $\imath\in G$ and $k(\imath)=1$ if $\imath\in\mathcal{I}\setminus G$. For $\imath\in\mathcal{I}\setminus G$, pick $s_\imath\in A_{\imath}$ and consider $\vec{s}=\langle s_\imath \rangle_{\imath\in\mathcal{I}}$. Then 

\begin{equation*}
\begin{split}
\gamma\cdot\big|{\bigtimes}_{\imath\in\mathcal{I}}F_{\imath,k(\imath)}\big| &> \big|\big({\bigtimes}_{\imath\in\mathcal{I}}A_\imath\big)\cap \big(\big({\bigtimes}_{\imath\in\mathcal{I}}F_{\imath,k(\imath)}\big)\cdot\vec{s}\big)\big|\\
&=\big|{\bigtimes}_{\imath\in\mathcal{I}}\big(A_\imath\cap\big(F_{\imath,k(\imath)}\cdot s_\imath \big)\big)\big|\\
&=\big|{\bigtimes}_{\imath\in G}\big(A_\imath\cap\big(F_{\imath,t(\imath)}\cdot s_\imath\big)\big)\big|\\
&> {\prod}_{\imath\in G}(\delta_\imath - \epsilon)\cdot \big|F_{\imath,t(\imath)}\big|\\
&> \gamma\cdot {\prod}_{\imath\in G}\big|F_{\imath,t(\imath)}\big|=\gamma\cdot \big|{\bigtimes}_{\imath\in\mathcal{I}}F_{\imath,k(\imath)}\big|,
\end{split}
\end{equation*}
a contradiction.
\end{proof}

\begin{lemma}\label{Lemma 5.12}
Let $\{S_\imath : \imath \in\mathcal{I}\}$ be a collection of semigroups with given nets $\mathcal{F}_\imath=\langle F_{\imath,j}\rangle_{j\in J_\imath}$ in $\mathcal{P}_f(S_\imath)$ for each $\imath\in\mathcal{I}$ and assume that each $S_\imath$ has an identity element. Let $\mathcal{I}=\mathcal{I}_1\sqcup\mathcal{I}_2$ be a partition of $\mathcal{I}$. Let $\mathcal{F}$, $\mathcal{G}$ and $\mathcal{H}$ be the product nets determined by $\{\mathcal{F}_\imath : \imath \in\mathcal{I}\}$, $\{\mathcal{F}_\imath : \imath \in\mathcal{I}_1\}$ and $\{\mathcal{F}_\imath : \imath \in\mathcal{I}_2\}$ in $\mathcal{P}_f\big(\bigtimes_{\imath\in\mathcal{I}}S_\imath\big)$, $\mathcal{P}_f\big(\bigtimes_{\imath\in\mathcal{I}_1}S_\imath\big)$ and $\mathcal{P}_f\big(\bigtimes_{\imath\in\mathcal{I}_2}S_\imath\big)$, respectively. Let $\sigma: \big(\bigtimes_{\imath\in\mathcal{I}_1}S_\imath\big)\times\big(\bigtimes_{\imath\in\mathcal{I}_2}S_\imath\big)\rightarrow \bigtimes_{\imath\in\mathcal{I}}S_\imath$ be the natural isomorphism defined by $\sigma \big(\langle s_\imath \rangle_{\imath\in\mathcal{I}_1},\langle s_\imath \rangle_{\imath\in\mathcal{I}_2}\big)=\langle s_\imath \rangle_{\imath\in\mathcal{I}}$ for $s_\imath\in S_\imath$, $\imath\in \mathcal{I}$. Let $A$ be a subset of $\big(\bigtimes_{\imath\in\mathcal{I}_1}S_\imath\big)\times\big(\bigtimes_{\imath\in\mathcal{I}_2}S_\imath\big)$. Then $d_{\mathcal{G}*\mathcal{H}}^*(A)=d_{\mathcal{F}}^*(\sigma[A])$.
\end{lemma}

\begin{proof}
Let $\mathcal{F}=\big\langle \bigtimes_{\imath\in\mathcal{I}}F_{\imath,j(\imath)}\big\rangle_{\vec{j}\in\mathfrak{J}}$, $\mathcal{G}=\big\langle \bigtimes_{\imath\in\mathcal{I}_1}F_{\imath,j_1(\imath)}\big\rangle_{\vec{j_1}\in\mathfrak{J}_1}$ and $\mathcal{H}=\big\langle \bigtimes_{\imath\in\mathcal{I}_2}F_{\imath,j_2(\imath)}\big\rangle_{\vec{j_2}\in\mathfrak{J}_2}$. Suppose first that $d_{\mathcal{G}*\mathcal{H}}^*(A)<d_{\mathcal{F}}^*(\sigma[A])$ and pick $\gamma$ such that $d_{\mathcal{G}*\mathcal{H}}^*(A) < \gamma < d_{\mathcal{F}}^*(\sigma[A])$. Pick $\big(\vec{j_1},\vec{j_2}\big)\in \mathfrak{J}_1 \times \mathfrak{J}_2$ such that whenever $\big(\vec{k_1},\vec{k_2}\big)\geq \big(\vec{j_1},\vec{j_2}\big)$ and $\big(\vec{s_1},\vec{s_2}\big)\in \big(\bigtimes_{\imath\in\mathcal{I}_1}S_\imath\big)\times\big(\bigtimes_{\imath\in\mathcal{I}_2}S_\imath\big)$, one has 
\begin{equation*}
\begin{split}
&\big|A\cap\big(\big(F_{\vec{k_1}}\times F_{\vec{k_2}}\big)\cdot\big(\vec{s_1},\vec{s_2}\big)\big)\big|<\gamma\cdot \big|F_{\vec{k_1}}\times F_{\vec{k_2}}\big|,\\
\text{i.e. }&\big|\sigma[A]\cap\big(\sigma\big[F_{\vec{k_1}}\times F_{\vec{k_2}}\big]\cdot\sigma\big(\vec{s_1},\vec{s_2}\big)\big)\big|<\gamma\cdot \big|\sigma\big[F_{\vec{k_1}}\times F_{\vec{k_2}}\big]\big|,
\end{split}
\end{equation*}
where $F_{\vec{k_1}}:=\bigtimes_{\imath\in\mathcal{I}_1}F_{\imath,k_1(\imath)}$ and $F_{\vec{k_2}}:=\bigtimes_{\imath\in\mathcal{I}_2}F_{\imath,k_2(\imath)}$. Now consider the vector $\vec{j}\in\mathfrak{J}$ defined by $j(\imath):=\begin{cases}
j_1(\imath) & \text{if }\imath\in\mathcal{I}_1\\
j_2(\imath) & \text{if }\imath\in\mathcal{I}_2
\end{cases}$. Then pick $\vec{k}\in\mathfrak{J}$ such that $\vec{k}\geq\vec{j}$ and pick $\vec{s}\in\bigtimes_{\imath\in\mathcal{I}}S_\imath$ such that 
$$\big|\sigma[A]\cap\big(\big({\bigtimes}_{\imath\in\mathcal{I}}F_{\imath,k(\imath)}\big)\cdot\vec{s}\big)\big|>\gamma\cdot\big|{\bigtimes}_{\imath\in\mathcal{I}}F_{\imath,k(\imath)}\big|.$$
Pick $\vec{s_1}\in\bigtimes_{\imath\in\mathcal{I}_1}S_\imath$ and $\vec{s_2}\in\bigtimes_{\imath\in\mathcal{I}_2}S_\imath$ such that $\sigma(\vec{s_1},\vec{s_2})=\vec{s}$. Also define $\vec{k_1}\in\mathfrak{J}_1$ by $k_1(\imath)=k(\imath)$ for each $\imath\in\mathcal{I}_1$, and define $\vec{k_2}\in\mathfrak{J}_2$ by $k_2(\imath)=k(\imath)$ for each $\imath\in \mathcal{I}_2$. Then we have $\big(\vec{k_1},\vec{k_2}\big)\geq \big(\vec{j_1},\vec{j_2}\big)$ and $\sigma\big[F_{\vec{k_1}}\times F_{\vec{k_2}}\big]=\bigtimes_{\imath\in\mathcal{I}}F_{\imath,k(\imath)}$, and thus 
\begin{equation*}
\begin{split}
\gamma\cdot\big|{\bigtimes}_{\imath\in\mathcal{I}}F_{\imath,k(\imath)}\big|&<\big|\sigma[A]\cap\big(\big({\bigtimes}_{\imath\in\mathcal{I}}F_{\imath,k(\imath)}\big)\cdot\vec{s}\big|\\
&=\big|\sigma[A]\cap\big(\sigma\big[F_{\vec{k_1}}\times F_{\vec{k_2}}\big]\cdot\sigma(\vec{s_1},\vec{s_2})\big)\big|\\
&<\gamma\cdot\big|\sigma\big[F_{\vec{k_1}}\times F_{\vec{k_2}}\big]\big|=\gamma\cdot\big|{\bigtimes}_{\imath\in\mathcal{I}}F_{\imath,k(\imath)}\big|,
\end{split}
\end{equation*}
a contradiction.\par
Now suppose that $d_{\mathcal{G}*\mathcal{H}}^*(A)>d_{\mathcal{F}}^*(A)$ and pick $\gamma$ such that $d_{\mathcal{G}*\mathcal{H}}^*(A)>\gamma > d_{\mathcal{F}}^*(A)$. Pick $\vec{j}\in\mathfrak{J}$ such that whenever $\vec{k}\geq\vec{j}$ and $\vec{s}\in\bigtimes_{\imath\in\mathcal{I}}S_\imath$, one has 
$$\big|\sigma[A]\cap\big(\big({\bigtimes}_{\imath\in\mathcal{I}}F_{\imath,k(\imath)}\big)\cdot\vec{s}\big)\big|<\gamma\cdot\big|{\bigtimes}_{\imath\in\mathcal{I}}F_{\imath,k(\imath)}\big|.$$
Now consider the vector $\vec{j_1}\in\mathfrak{J}_1$ defined by $j_1(\imath)=j(\imath)$ for each $\imath\in\mathcal{I}_1$, and consider the vector 
$\vec{j_2}\in\mathfrak{J}_2$ defined by $j_2(\imath)=j(\imath)$ for each $\imath\in\mathcal{I}_2$. Then pick $(\vec{k_1},\vec{k_2})\in \mathfrak{J}_1\times\mathfrak{J}_2$ such that $(\vec{k_1},\vec{k_2})\geq (\vec{j_1},\vec{j_2})$ and pick $(\vec{s_1},\vec{s_2})\in\big(\bigtimes_{\imath\in\mathcal{I}_1}S_\imath\big)\times\big(\bigtimes_{\imath\in\mathcal{I}_2}S_\imath\big)$ such that 
\begin{equation*}
\begin{split}
&\big|A\cap\big(\big(F_{\vec{k_1}}\times F_{\vec{k_2}}\big)\cdot (\vec{s_1},\vec{s_2})\big)\big|>\gamma\cdot\big|F_{\vec{k_1}}\times F_{\vec{k_2}}\big|,\\
\text{i.e. }&\big|\sigma[A]\cap\big(\sigma\big[F_{\vec{k_1}}\times F_{\vec{k_2}}\big]\cdot \sigma(\vec{s_1},\vec{s_2})\big)\big|>\gamma\cdot\big|\sigma\big[F_{\vec{k_1}}\times F_{\vec{k_2}}\big]\big|,
\end{split}
\end{equation*}
where $F_{\vec{k_1}}:=\bigtimes_{\imath\in\mathcal{I}_1}F_{\imath,k_1(\imath)}$ and $F_{\vec{k_2}}:=\bigtimes_{\imath\in\mathcal{I}_2}F_{\imath,k_2(\imath)}$. Let $\vec{s}:=\sigma(\vec{s_1},\vec{s_2})$ and consider the vector $\vec{k}\in\mathcal{I}$ where  $k(\imath):=\begin{cases}
k_1(\imath) &\text{ if } \imath\in\mathcal{I}_1\\
k_2(\imath) &\text{ if } \imath\in\mathcal{I}_2
\end{cases}
$. Then we have $\vec{k}\geq \vec{j}$ and $\sigma\big[F_{\vec{k_1}}\times F_{\vec{k_2}}\big]=\bigtimes_{\imath\in\mathcal{I}}F_{\imath,k(\imath)}$, and thus 
\begin{equation*}
\begin{split}
\gamma\cdot\big|{\bigtimes}_{\imath\in\mathcal{I}}F_{\imath,k(\imath)}\big|&=\gamma\cdot\big|\sigma\big[F_{\vec{k_1}}\times F_{\vec{k_2}}\big]\big|\\
&<\big|\sigma[A]\cap\big(\sigma\big[F_{\vec{k_1}}\times F_{\vec{k_2}}\big]\cdot\sigma(\vec{s_1},\vec{s_2})\big)\big|\\
&=\big|\sigma[A]\cap\big(\big({\bigtimes}_{\imath\in\mathcal{I}}F_{\imath,k(\imath)}\big)\cdot\vec{s}\big)\big|<\gamma\cdot\big|{\bigtimes}_{\imath\in\mathcal{I}}F_{\imath,k(\imath)}\big|,
\end{split}
\end{equation*}
a contradiction.
\end{proof}

\begin{remark}
In the next few concluding results of this section, whenever we consider an infinite collection of semigroups, we assume each of them to be right cancellative, so that their Cartesian product also becomes right cancellative. Note that this may not be true if we restrict to the class of $b$-weakly right cancellative semigroups (for some $b\in\mathbb{N}$).
\end{remark}

In particular, Proposition \ref{Proposition 5.7} implies the following result.

\begin{corollary}\label{Corollary 5.14} 
Let $\{S_\imath:\imath\in\mathcal{I}\}$ be an infinite collection of right cancellative semigroups with given nets $\mathcal{F}_{\imath}$ in $\mathcal{P}_f(S_\imath)$ for each $\imath\in\mathcal{I}$ and assume that each $S_\imath$ has an identity element. Let $\mathcal{F}$ be the product net determined by $\{\mathcal{F}_\imath:\imath\in\mathcal{I}\}$. Let $\pi_{\imath_{0}}:\bigtimes_{\imath\in \mathcal{I}}S_{\imath}\rightarrow S_{\imath_{0}}$ be the projection map onto the $\imath_{0}^{th}$ component. Then $\pi_{\imath_{0}}$ maps $D$-sets (with respect to the product net $\mathcal{F}$) of $\bigtimes_{\imath\in\mathcal{I}}S_{\imath}$ onto the $D$-sets (with respect to $\mathcal{F}_{\imath_0}$) of $S_{\imath_{0}}$.
\end{corollary}

\begin{proof}
By Lemma \ref{Lemma 5.11}, $d_{\mathcal{F}}^{*}(\pi^{-1}(B))=d_{\mathcal{F}_{i_{0}}}^{*}(B)$ for each $B\subseteq S_{\imath_0}$. Hence the conclusion follows from Proposition \ref{Proposition 5.7}.
\end{proof}

We shall utilize the following notion of thick set \cite[Definition 1.2(a)]{hs-2} and its algebraic characterization \cite[Lemma 1.4(d)]{hs-2} to get a partial converse of Corollary \ref{Corollary 5.14}.  

\begin{definition}
Let $S$ be a semigroup. A subset $A$ of $S$ is thick if for each $F\in \mathcal{P}_f(S)$ there exists $s\in S$ such that $Fs\subseteq A$.
\end{definition}

\begin{lemma}\label{Lemma 5.16}
Let $S$ be a semigroup. Then a subset $A$ of $S$ is thick if and only if there is a closed left ideal $L$ of $\beta S$ such that $L\subseteq \overline{A}$.
\end{lemma}

\begin{lemma}\label{Lemma 5.17}
Let $b\in\mathbb{N}$ and assume that $S$ is a $b$-weakly right cancellative semigroup. Let $A$ be a thick subset of $S$. Then $A$ is a $D$-set with respect to any net $\mathcal{F}$ in $\mathcal{P}_f(S)$.
\end{lemma}

\begin{proof}
Pick a closed left ideal $L$ of $\beta S$ such that $L\subseteq \overline{A}$. Pick a net $\mathcal{F}$ in $\mathcal{P}_f(S)$ and consider the closed right ideal $D_{\mathcal{F}}^*$. Since every left ideal contains a minimal left ideal, we may assume that $L$ is a minimal left ideal. Pick a minimal right ideal $R\subseteq D_{\mathcal{F}}^*$. Then $L\cap R$ is a group, so pick an idempotent $p\in L\cap R$. Then $p\in \overline{A}\cap D_{\mathcal{F}}^*$ and therefore $A$ is a $D$-set with respect to $\mathcal{F}$. 
\end{proof}

The following theorem is a partial converse of Corollary \ref{Corollary 5.14}.

\begin{theorem}\label{Theorem 5.18}
Let $\{S_\imath:\imath\in\mathcal{I}\}$ be an infinite collection of right cancellative semigroups with given nets $\mathcal{F}_{\imath}=\langle F_{\imath,j}\rangle_{j\in J_{\imath}}$ in $\mathcal{P}_f(S_\imath)$ for each $\imath\in\mathcal{I}$ and assume that each $S_\imath$ has an identity element. Let $\mathcal{F}$ be the product net determined by $\{\mathcal{F}_\imath:\imath\in\mathcal{I}\}$. If $A_\imath$ is a $D$-set (with respect to $\mathcal{F}_{\imath}$) in $S_{\imath}$ for each $\imath\in\mathcal{I}$ and the $A_{\imath}$'s are thick for all but finitely many $\imath$'s, then $\bigtimes_{\imath\in\mathcal{I}}A_{\imath}$ is a $D$-set (with respect to the product net $\mathcal{F}$) in $\bigtimes_{\imath\in\mathcal{I}}S_{\imath}$.
\end{theorem}

\begin{proof}
Let $\mathcal{I}_1=\{\imath\in\mathcal{I}:A_\imath\subseteq S_\imath \text{ is thick}\}$. Let $\mathcal{I}_2:=\mathcal{I}\setminus\mathcal{I}_1$. Now $\bigtimes_{\imath\in\mathcal{I}_1}A_\imath$ is a thick subset of the semigroup $\bigtimes_{\imath\in\mathcal{I}_1}S_\imath$. So by Lemma \ref{Lemma 5.17}, $\bigtimes_{\imath\in\mathcal{I}_1}A_\imath$ is a $D$-set in the semigroup $\bigtimes_{\imath\in\mathcal{I}_1}S_\imath$ with respect to the product net $\mathcal{G}$ determined by $\{\mathcal{F}_\imath:\imath\in\mathcal{I}_1\}$. On the other hand, as $\mathcal{I}_2$ is a finite set, so by Remark \ref{Remark 5.9} $\bigtimes_{\imath\in\mathcal{I}_2}A_\imath$ is a $D$-set in the semigroup $\bigtimes_{\imath\in\mathcal{I}_2}S_\imath$ with respect to the product net $\mathcal{H}$ determined by $\{\mathcal{F}_\imath:\imath\in\mathcal{I}_2\}$. Hence by Theorem \ref{Theorem 5.8}, $\big(\bigtimes_{\imath\in\mathcal{I}_1}A_\imath\big)\times\big(\bigtimes_{\imath\in\mathcal{I}_2}A_2\big)$ is a $D$-set in $\big(\bigtimes_{\imath\in\mathcal{I}_1}S_\imath\big)\times\big(\bigtimes_{\imath\in\mathcal{I}_2}S_2\big)$ with respect to the product net $\mathcal{G}*\mathcal{H}$. Therefore by Lemma \ref{Lemma 5.12}, $\bigtimes_{\imath\in\mathcal{I}}A_\imath$ is a $D$-set in the semigroup $\bigtimes_{\imath\in\mathcal{I}}S_\imath$ with respect to the product net $\mathcal{F}$.
\end{proof}

\begin{remark}
In general, the infinite product of $D$-sets may not be a $D$-set. For example, by \cite[Lemma 5.19.1]{hs} $2\mathbb{Z}$ is a $D$-set in the additive group $(\mathbb{Z},+)$, and the upper Banach density of $2\mathbb{Z}$ is  $d_{\mathcal{F}}^*(2\mathbb{Z})=\frac{1}{2}$ (with respect to the F{\o}lner sequence $\mathcal{E}=\langle E_n \rangle_{n\in \mathbb{N}}$, where for each $n\in\mathbb{N}$, $E_n=\{1,2,...,n\}$). But by Lemma \ref{Lemma 5.11}, with respect to the product net $\mathcal{F}$ on the countable product $\mathbb{Z}\times\mathbb{Z}\times ...,\,d_{\mathcal{F}}^*(2\mathbb{Z}\times 2\mathbb{Z} \times ...)=0$. Hence $2\mathbb{Z}\times 2\mathbb{Z}\times ...$ is not a $D$-set.
\end{remark}

\section{Relations with Central sets and $C$-sets}\label{Section 6}

As mentioned in the Introduction, for the additive semigroup $\mathbb{N}$ we have the following inclusions.
$$\text{ Central Sets }\subset D\text{-sets} \subset C\text{-sets}.$$
In this section, we study these inclusions in general with some natural restrictions on the semigroup $S$, and the corresponding net $\mathcal{F}$ in $\mathcal{P}_f(S)$. We shall begin by recalling the definition of central sets in an arbitrary semigroup \cite[Definition 4.42]{hs}.

\begin{definition}
Let $S$ be a semigroup and let $A\subseteq S$. Then $A$ is central in $S$ if there is some idempotent $p\in K(\beta{S})$ such that $A\in p$. 
\end{definition}

The following proposition is concerning the implication of central sets to $D$-sets.

\begin{proposition}\label{Proposition 6.2}
Let $b\in\mathbb{N}$ and assume that $S$ is both a $b$-weakly left and a $b$-weakly right cancellative semigroup. Let $\mathcal{F}$ be a net in $\mathcal{P}_f(S)$ which satisfies the $(*)$ condition. Then every central set in $S$ is a $D$-set in $S$ with respect to $\mathcal{F}$.
\end{proposition}

\begin{proof}
If $S$ is both $b$-weakly left and $b$-weakly right cancellative for some $b\in\mathbb{N}$, and if $\mathcal{F}$ satisfies the $(*)$ condition, then $D_{\mathcal{F}}^{*}$ is a two sided ideal. Hence $D_{\mathcal{F}}^{*}$ contains $\mathcal{K}(\beta S)$. Therefore, any idempotent of $\mathcal{K}(\beta S)$ is also an idempotent of $D_{\mathcal{F}}^{*}$. 
\end{proof}

To study the inclusion of $D$-sets inside $C$-sets, let us begin by recalling the definition of $C$-set in an arbitrary semigroup \cite[Definition 3.1]{j-1}.

\begin{definition}
Let $S$ be a semigroup.
\begin{enumerate}
\item For each positive integer $m$ put $\mathcal{J}_m=\{(t_1,t_2,...,t_m)\in\mathbb{N}^m:\,t_1<t_2<...<t_m\}$.
\item Put $\mathcal{T}={}^{\mathbb{N}}S$, and if the semigroup is clear from context, we will instead write $\mathcal{T}$ for $\mathcal{T}(S)$.
\item For all positive integers $m$, every $a\in S^{m+1}$, every $t\in \mathcal{J}_m$, and for all $f\in \mathcal{T}$, put $x(m,a,t,f)=\Big(\displaystyle{\prod_{i=1}^m(a(i)f(t_i))}\Big)a(m+1)$.
\item We call $A\subseteq S$ a $C$-set if there exists $m:\mathcal{P}_f(\mathcal{T})\rightarrow \mathbb{N},\,\alpha \in \bigtimes_{F\in\mathcal{P}_f(\mathcal{T})}S^{m(F)+1}$, and $\tau\in \bigtimes_{F\in\mathcal{P}_f(\mathcal{T})}\mathcal{J}_{m(F)}$ such that the following two conditions are satisfied:
\begin{enumerate}
\item If $F$ and $G$ are both elements of $\mathcal{P}_f(\mathcal{T})$ with $F\subsetneq G$, then $\tau (F)\big(m(F)\big)<\tau (G)(1)$, and 
\item Whenever $m$ is a positive integer, $G_1,G_2,...,G_m\in\mathcal{P}_f(\mathcal{T})$ with $G_1\subsetneq G_2 \subsetneq ...\subsetneq G_m$, and for each $i\in\{1,2,...,m\}$ we have $f_i\in G_i$, then $\displaystyle{\prod_{i=1}^m x\big(m(G_i),\alpha (G_i),\tau (G_i),f_i\big)\in A}$.
\end{enumerate}
\end{enumerate} 
\end{definition}

We shall also recall the definition of $J$-sets \cite[Definition 2.1 and 2.6]{j-1}, which are closely related to $C$-sets and will be used later.

\begin{definition}
Let $S$ be a semigroup. 
\begin{enumerate}
\item We call $A\subseteq S$ a $J$-set (in $S$) if for every $F\in\mathcal{P}_f(\mathcal{T})$, there exist $m\in\mathbb{N},\,a\in S^{m+1}$ and $t\in\mathcal{J}_m$ such that for every $f\in F$, we have $x(m,a,t,f)\in A$.
\item $J(S)=\{p\in\beta S: \text{ For every } A\in p \text{ we have that } A \text{ is a } J\text{-set}\}$.
\end{enumerate}
\end{definition}

If $S$ is a commutative semigroup, then we know that $S$ satisfies the Strong F{\o}lner Condition (SFC) \cite[Theorem 4]{aw}. The following definition for (SFC) is taken from \cite{hs-2}.

\begin{definition}
A semigroup $S$ satisfies the Strong F{\o}lner Condition (SFC) if $(\forall H\in\mathcal{P}_f(S))(\forall \epsilon >0)(\exists K\in \mathcal{P}_f(S))(\forall s\in H)(|K\triangle sK|<\epsilon\cdot |K|)$.
\end{definition}

In a semigroup $S$, (SFC) ensures the existence of a special type of net in $\mathcal{P}_f(S)$, called F{\o}lner net \cite[Theorem 4.2]{hs-2}. The definition of F{\o}lner net is as follows \cite[Definition 4.1]{hs-2}.

\begin{definition}
Let $S$ be a semigroup and let $\mathcal{F}=\langle F_i\rangle_{i \in I}$ be a net in $\mathcal{P}_f(S)$. Then $\mathcal{F}$ is a F{\o}lner net if for each $s\in S$, the net $\Big\langle\frac{|sF_i\triangle F_i|}{|F_i|}\Big\rangle_{i\in I}$ converges to $0$.
\end{definition}

The following notion of F{\o}lner density \cite[Definition 4.15]{hs-2} will be used in the next lemma to provide an upper bound for the Upper Banach density with respect to any F{\o}lner net.

\begin{definition}
Let $S$ be a semigroup which satisfies the (SFC). Let $A\subseteq S$. Then the F{\o}lner density of $A$ is  $d_{F\phi}(A)=\sup \{\alpha:(\forall H\in\mathcal{P}_f(S))(\forall \epsilon >0)(\exists K\in \mathcal{P}_f(S))(\forall s\in H)(|K\triangle sK|<\epsilon\cdot |K|) \text{ and } |A\cap K|\geq\alpha\cdot |K|\}$.
\end{definition} 

\begin{lemma}\label{Lemma 6.8}
Let $S$ be a semigroup which satisfies (SFC). Let $\mathcal{F}$ be a F{\o}lner net in $\mathcal{P}_f(S)$. Then for all $A\subseteq S$, $d_{\mathcal{F}}^*(A)\leq d_{F\phi}(A)$. 
\end{lemma}
\begin{proof}
This is an immediate consequence of \cite[Theorem 4.16]{hs-2} and the trivial fact that $d_{\mathcal{F}}^*(A)\leq \overline{d}_{\mathcal{F}}(A)$.
\end{proof}

So, given a F{\o}lner net $\mathcal{F}$ in $\mathcal{P}_f(S)$, we have $D_{\mathcal{F}}^*\subseteq \{p\in\beta S:(\forall A\in p)(d_{F\phi}(A)>0)\}$ and this insists us to define the following terminology.

\begin{definition}
Let $S$ be a semigroup which satisfies (SFC). Let $\triangle (S)=\{p\in\beta S:(\forall A\in p)(d_{F\phi}(A)>0)\}$. We say that $A\subseteq S$ is a $D_{F\phi}$-set if it is a member of an Idempotent in $\triangle (S)$.
\end{definition}

Thus, in a commutative semigroup $S$, we conclude that every $D_{F\phi}$-set is a $C$-set. This follows from \cite[Lemma 2.2 and Theorem 6.10]{hs-3} and \cite[Theorem 3.2]{j-1}, which we recall here in order.

\begin{lemma}
Let $S$ be a semigroup which satisfies (SFC). Then for every $A\subseteq S$, there exists a left invariant mean $\mu$ on $S$ (i.e. a regular Borel probability measure $\mu$ on $\beta{S}$ such that $\mu(\overline{s^{-1}B})=\mu(\overline{B})$ for every $B\subseteq S$, and every $s\in S$) such that $d_{F\phi}(A)=\mu(\overline{A})$.
\end{lemma}

\begin{theorem}\label{Theorem 6.11}
Let $S$ be a commutative semigroup and let $A\subseteq S$ be such that $\mu (\overline{A})>0$ for some left invariant mean $\mu$ on $S$. Then $A$ is a $J$-set.
\end{theorem}

\begin{theorem}
Let $S$ be a semigroup. If $p\in J(S)$ is an idempotent, then every $A\in p$ is a $C$-set.
\end{theorem}

Finally, in a commutative semigroup, the above discussion culminates into the following theorem.

\begin{theorem}\label{Theorem 6.13}
Let $S$ be a commutative semigroup with a given F{\o}lner net $\mathcal{F}$ in $\mathcal{P}_f(S)$. Then any $D$-set with respect to $\mathcal{F}$ is a $C$-set.
\end{theorem}

Note that the only obstruction to prove the Theorem \ref{Theorem 6.13} for a general class of noncommutative semigroup is the Theorem \ref{Theorem 6.11}, which uses an Ergodic Szemer\'edi theorem for commuting IP-system \cite[Theorem A]{fk} as one of the main ingredients. At the end, with the support of Theorem 6.13, we raise the following question.

\begin{question}
Let $S$ be a semigroup which satisfies the Strong F{\o}lner Condition (SFC). Let $F$ be a F{\o}lner
net in $P_f(S)$. Then does a $D$-set with respect to $F$ always become a C-set?
\end{question}

\section*{Acknowledgments}
The authors are indebted to the anonymous reviewers for their generous comments and suggestions on
the previous manuscripts. The first author gratefully acknowledges the UGC-NET SRF fellowship with Ref.
No. 20/12/2015(ii)EU-V of CSIR-UGC NET December 2015. The third author acknowledges the fellowship
from IISER Berhampur with Ref. No. IISERBPR/CoFA/2019/74.

\end{document}